\documentclass[reqno]{amsart}
\usepackage[]{latexsym}
\usepackage[]{amssymb}
\usepackage[]{amsmath}
\usepackage{amsfonts}
\usepackage{amsthm}
\usepackage[all]{xy}
\usepackage[]{mathrsfs}
\usepackage{enumerate}
\usepackage{color}
\usepackage[capitalise]{cleveref}

\newtheorem{theorem}{Theorem}[section]

\newtheorem{thm}[theorem]{Theorem}
\newtheorem{cor}[theorem]{Corollary}
\newtheorem{prop}[theorem]{Proposition}
\newtheorem{lemma}[theorem]{Lemma}

\theoremstyle{definition}
\newtheorem{ex}[theorem]{Example}

\numberwithin{equation}{theorem}

\DeclareMathAlphabet{\mathpzc}{OT1}{pzc}{m}{it}

\DeclareMathOperator{\kar}{char}
\begin{document}

\newcommand{\str}{{\mathrm{sT}}}
\newcommand{\Trd}{{\mathrm{Trd}}}
\newcommand{\rad}{{\mathrm{rad}}}
\newcommand{\id}{{\mathrm{id}}}
\newcommand{\Ad}{{\mathrm{Ad}}}
\newcommand{\Ker}{{\mathrm{Ker}}}
\newcommand{\wedges}[1]{\d b_1\wedge\ldots\wedge \d b_{#1}}
\newcommand{\rank}{{\mathrm{rank}}}
\renewcommand{\dim}{{\mathrm{dim}}}
\newcommand{\coker}{{\mathrm{Coker}}}
\newcommand{\can}{\overline{\rule{2.5mm}{0mm}\rule{0mm}{4pt}}}
\newcommand{\End}{{\mathrm{End}}}
\newcommand{\Sand}{{\mathrm{Sand}}}
\newcommand{\Hom}{{\mathrm{Hom}}}
\newcommand{\Nrd}{{\mathrm{Nrd}}}
\newcommand{\Srd}{{\mathrm{Srd}}}
\newcommand{\ad}{{\mathrm{ad}}}
\newcommand{\rk}{{\mathrm{rk}}}
\newcommand{\Mon}{{\mathrm{Mon}}}
\newcommand{\disc}{{\mathrm{disc}}}
\newcommand{\Sym}{{\mathrm{Sym}}}
\newcommand{\Skew}{{\mathrm{Skew}}}
\newcommand{\Nrp}{{\mathrm{Nrp}}}
\newcommand{\Trp}{{\mathrm{Trp}}}
\newcommand{\Alt}{{\mathrm{Alt}}}
\newcommand{\Symd}{{\mathrm{Symd}}}
\renewcommand{\dir}{{\mathrm{dir}}}
\renewcommand{\geq}{\geqslant}
\renewcommand{\leq}{\leqslant}
\newcommand{\an}{{\mathrm{an}}}
\newcommand{\alt}{{\mathrm{alt}}}
\renewcommand{\Im}{{\mathrm{Im}}}
\newcommand{\Int}{{\mathrm{Int}}}
\renewcommand{\d}{{\mathrm{d}}}
\newcommand{\qf}[1]{\mbox{$\langle #1\rangle $}}
\newcommand{\pff}[1]{\mbox{$\langle\!\langle #1
\rangle\!\rangle $}}
\newcommand{\pfr}[1]{\mbox{$\langle\!\langle #1 ]]$}}
\newcommand{\HH}{{\mathbb H}}
\newcommand{\s}{{\sigma}}
\newcommand{\lra}{{\longrightarrow}}
\newcommand{\ZZ}{{\mathbb Z}}
\newcommand{\NN}{{\mathbb N}}
\newcommand{\FF}{{\mathbb F}}
\newcommand{\PERP}{\mbox{\raisebox{-.5ex}{{\Huge $\perp$}}}}
\newcommand{\Perp}{\mbox{\raisebox{-.2ex}{{\Large $\perp$}}}}
\newcommand{\M}[1]{\mathbb{M}( #1)}
\newcommand{\qi}{\mid}
\newcommand{\qil}{\,\cdot\!\!\mid}
\newcommand{\qir}{\mid\!\!\cdot\,}
\newcommand{\qilr}{\,\cdot\!\!\mid\!\!\cdot\,}
\newcommand{\qp}[2]{\mbox{$[ {#1}\,|\!|\,{#2} )$}}

\newcommand{\vf}{\varphi}
\newcommand{\mg}[1]{#1^{\times}}

\title{Totally decomposable  quadratic pairs}
\author{Karim Johannes Becher}
\email{KarimJohannes.Becher@uantwerpen.be, Andrew.Dolphin@uantwerpen.be}
\address{Universiteit Antwerpen, Departement Wiskunde-Informatica, Middelheim\-laan~1, 2020 Antwerpen, Belgium.}
\author{Andrew Dolphin}
%\address{Departement Wiskunde-Informatica, Universiteit Antwerpen, Belgium}
\thanks{This work was supported by the {Deutsche Forschungsgemeinschaft} (project \emph{The Pfister Factor Conjecture in Characteristic Two}, BE 2614/4), the FWO Odysseus programme (project \emph{Explicit Methods in Quadratic Form Theory}), the {Zukunftskolleg, Universit\"at Konstanz} and the Fonds Sp\'eciaux de Recherche (postdoctoral funding) through the Universit\'e Catholique de Louvain.}
%\thanks{2010 \emph{Mathematics Subject Classification}. Primary 11E39; Secondary 11E81; 12F05; 12F10. \emph{Key words and phrases.} Central simple algebras; involutions; quadratic pairs; characteristic two; unitary involutions; Pfister Factor Conjecture. }

\begin{abstract} In this paper we show that a  split central simple algebra with quadratic pair   which decomposes into a tensor product of quaternion algebras with involution and a quaternion algebra with quadratic pair is adjoint to a quadratic Pfister form. 
This result is new in characteristic two, otherwise it is equivalent to the Pfister Factor Conjecture proven in  \cite{Becher:qfconj}.


\medskip\noindent
\emph{Mathematics Subject Classification (MSC 2010):} 11E39, 11E81, 16W10, 16K20. %12F05, 12F10. 
\end{abstract}

\maketitle

\section{Introduction}

To every symmetric or alternating bilinear form, and hence to every quadratic form when the base field is not of characteristic two, we can associate a central simple algebra with involution. In this way,  the theory of quadratic forms is embedded into the larger theory of algebras with involution, and through the use of this correspondence,  quadratic form theory has provided  inspiration for the study of algebras with involution (see \cite{Knus:1998}).
Pfister forms are a central concept in the modern algebraic theory of quadratic forms. It is therefore natural to look for class of central simple algebras with involution  which extends the notion of a Pfister form. This question was first raised in  \cite{parimala:pfisterinv}. Involutions adjoint to Pfister forms are tensor products of quaternion algebras with involution.  Thus,  tensor products of quaternion algebras with involution are a natural candidate. 

%In  \cite{Becher:qfconj} it is shown that, over fields of characteristic different from two, 
 %a split algebra with totally decomposable orthogonal involution is adjoint to a Pfister form.  This result is known as the Pfister Factor Conjecture. 
 
 In characteristic two, quadratic forms and symmetric bilinear forms are not equivalent objects. The relation between bilinear Pfister forms and totally decomposable involutions in characteristic two was studied in \cite{dolphin:orthpfist}. 
In order to have an  object defined on a central simple algebra that corresponds to a quadratic form after splitting, the notion of a quadratic pair was introduced in  \cite[\S5]{Knus:1998}. In particular, one may use quadratic pairs to give an intrinsic definition of  twisted orthogonal groups in a manner that includes fields of characteristic two (see  \cite[\S23.B]{Knus:1998}).

Algebras with quadratic pair associated to  quadratic Pfister forms are  tensor products of quaternion algebras with involution and a quaternion algebra with quadratic pair. 
One may ask whether all such totally decomposable quadratic pairs on a split central simple algebra  
 are adjoint to a quadratic Pfister form.
%  Our main result shows that over fields of characteristic two, the answer to this question is positive. 
In characteristic different from two, where quadratic pairs are equivalent to orthogonal involutions, this is known to hold by the main result of  \cite{Becher:qfconj},  which says that in this case a
totally decomposable orthogonal involution on a split algebra
 is adjoint to a Pfister form. %  (see  \cite[Chapter 9]{Shapiro:2000} for background on this result). 
In this article we prove the corresponding result for quadratic pairs over fields of characteristic two.

   Our approach is particular to fields of characteristic two.  It allows us to capture more information on the resulting quadratic Pfister form (see \cref{cor:main}), in a way that  is not possible in general over fields of characteristic different from two (see \cref{ex:countex}). 
 This approach uses the unusual  properties of totally decomposable involutions in characteristic two found in \cite{dolphin:orthpfist}. In particular, we use that the isotropy behaviour of a totally decomposable orthogonal involution can be captured in the isotropy behaviour of an associated  bilinear Pfister form  (see \cref{thm:pfisterinvar}).

%Hence, while our main result holds for arbitrary characteristic, our approach to prove it in the remaining case of characteristic two does not.
The method used in \cite{Becher:qfconj} in characteristic different from two is based on a ramification exact sequence for Witt groups of quadratic forms over function fields of conics and the excellence property of these function fields.
The latter result is known to extend   to the case of arbitrary characteristic, but the former is  not yet available in characteristic two in a suitable form.
We intend to give a characteristic free proof in a future article.
Our solution to the missing case of characteristic two involves several  basic properties of quadratic pairs, %(in particular those in \S\ref{tenprodqps}),
 which we did not find explicitly in the literature. Following \cite[\S5]{Knus:1998},  we  present such statements without assumptions on the characteristic, for  ease of future reference.

\section{Quadratic forms over fields}

In this section we recall the  terminology and results we use from quadratic form theory. 
We refer to \cite[Chapters 1 and 2]{Elman:2008} as a general reference on symmetric bilinear and quadratic forms and for any basic notation and concepts not defined here. %Recall that over fields of characteristic different from $2$ the notion of a symmetric bilinear form and the notion of a quadratic form are equivalent (\cite[\S 7 p.41]{Elman:2008}). This is not the case over fields of characteristic $2$. 
For two objects $\alpha$ and $\beta$ in a certain category, we write $\alpha\simeq\beta$ to indicate that they are {isomorphic}, i.e.~that there exists an isomorphism between them.
This applies in particular to algebras with involution or with quadratic pair, but also to quadratic  and bilinear  forms, where the corresponding isomorphisms are  called {isometries}.
Throughout, let $F$ be a field. Let $\kar(F)$ denote the characteristic of $F$ and let $F^\times$ denote the multiplicative group of $F$. 

A  \emph{bilinear form over $F$} is a pair $(V,b)$ where $V$ is a finite-dimensional $F$-vector space and $b$ is a  $F$-bilinear map  $b:V\times V\rightarrow F$.  
 The \emph{radical of $(V,b)$} is the set $\rad(V,b)=\{x\in V\mid b(x,y)=0 \textrm{ for all } y\in  V\}$.
 We say that $(V,b)$ is
\emph{degenerate} if $\rad(V,b)\neq \{0\}$, and  \emph{nondegenerate} otherwise. 
Let $\varphi=(V,b)$ be a bilinear form over $F$. 
%We refer to $\dim_F(V)$ as the \emph{dimension of $\varphi$} and also denote it by $\dim(\varphi)$.
We say that $\vf$ is \emph{symmetric} if $b(x,y)=b(y,x)$ for all $x,y\in V$. 
 We call $\varphi$ \emph{alternating} if  $b(x,x)=0$ for all $x\in V$. 
 %If $\varphi$ is alternating  then we have that $b(x,y)=-b(y,x)$ for all $x,y\in V$, that is, $\varphi$ is \emph{skew-symmetric}. In particular every alternating  bilinear form over a field of characteristic $2$ is symmetric.
%We say a bilinear form $\varphi=(V,b)$ \emph{represents an element $a\in F$ } if there exists and $x\in V\backslash\{0\}$ such that $b(x,x)=a$, and we denote the set of elements in $F^\times$ represented by $\varphi$ by $D(\varphi)$.
The  form $\varphi$ is said to be \emph{isotropic} if there exists an $x\in V\backslash\{0\}$ such that $b(x,x)=0$, and \emph{anisotropic} otherwise. 
We call a subspace $W\subseteq V$ \emph{totally isotropic} (with respect to $b$) if $b|_{W\times W}=0$. 
If  $\varphi$ is nondegenerate and there exists a totally isotropic subspace $W\subseteq V$ such that $\dim_F(W)=\frac{1}{2}\dim_F(V)$, then we call $\vf$ \emph{metabolic}. %Note that an alternating form is always metabolic (see \cite[Chapters 1 and 2]{Elman:2008}).  
%Let $K/F$ be a field extension. Then we write $(V,b)_K=(V\otimes_F K, b_K) $ where $b_K$ is  given by $b_K( v\otimes k, v'\otimes k' )= kk'b(v,v')$ for all $v,v'\in V$ and $k,k'\in K$.

%Let  $\varphi_1=(V,b_1)$ and $\varphi_2=(W,b_2)$ be symmetric or alternating bilinear forms over $F$.
%The \emph{orthogonal sum of $\varphi_1$ and $\varphi_2$}, denoted $\varphi_1\perp \varphi_2$, is defined to be the pair $(V\times W,b'')$ with the $F$-bilinear  map $b'':(V\times W)\times (V\times W)\rightarrow F$ given by  $b'((v_1, w_1),(v_2, w_2))= b_1(v_1,v_2) + b_2(w_1,w_2)$ for all $v_1, v_2\in V$ and $w_1,w_2\in W$.
%The \emph{tensor product of  $\varphi_1$ and $\varphi_2$}, denoted $\vf_1\otimes \vf_2$, is defined to be the pair $(V\otimes W,b')$ with the $F$-bilinear  map $b':(V\otimes W)\times (V\otimes W)\rightarrow F$ determined by the rule that  $b'(v_1\otimes w_1,v_2\otimes w_2)= b_1(v_1,v_2) \cdot b_2(w_1,w_2)$ for all $v_1, v_2\in V$ and $w_1,w_2\in W$.
%By \cite[(1.17)]{Elman:2008}, for any nondegenerate non-alternating symmetric  bilinear form $\varphi$ over $F$, there exists $a_1,\ldots,a_n\in F^\times$ such that $\varphi\simeq \qf{a_1,\ldots,a_n}_b$.

%Let $K/F$ be a field extension. Then we write $(V,b)_K=(V\otimes_F K, b_K) $ where $b_K$ is  given by $b_K( v\otimes k , w\otimes k')= kk'b(v,w)$ for all $v,w\in V$ and $k,k'\in K$.  %This induces a natural inclusion $W_q(F)\rightarrow W_q(K)$.

For $a_1,\ldots,a_n\in F^\times$ the symmetric $F$-bilinear map $b:F^n\times F^n\rightarrow F$ given by $(x,y)\mapsto \sum_{i=1}^n a_ix_iy_i$ yields a
symmetric bilinear form $(F^n,b)$ over $F$,  which we denote by $\qf{a_1,\ldots,a_n}$.
For a positive integer $m$, by an \emph{$m$-fold bilinear Pfister form over $F$} we mean a nondegenerate symmetric bilinear form over $F$ that is isometric to
 $\qf{1,a_1}\otimes\ldots\otimes\qf{1,a_m}$  for some 
 $a_1,\ldots,a_m\in F^\times$.  We call $\qf{1}$  the \emph{$0$-fold bilinear Pfister form}.
%For $a_1,\ldots,a_m\in F^\times$, we denote the $m$-fold Pfister form by $\pff{a_1,\ldots, a_m}$. 
By  \cite[(6.3)]{Elman:2008},  a bilinear Pfister form  is either anisotropic or metabolic.

By a \textit{quadratic form over $F$} we  mean a pair $(V,q)$ of a finite-dimensional $F$-vector space $V$ and a map  $q:V\rightarrow F$ such that, firstly,  $q(\lambda x)=\lambda^2q(x)$ holds for all $x\in V$ and $\lambda\in F$, and secondly,  the map   $b_q:V\times V\rightarrow F\,,\,(x,y)\longmapsto q(x+y)-q(x)-q(y)$ is  $F$-bilinear.
Then $(V,b_q)$ is a symmetric bilinear form over $F$, called the \emph{polar form of $(V,q)$}. If $b_q$ is nondegenerate, we say that $(V,q)$ is \emph{nonsingular}, otherwise we say that $(V,q)$ is \emph{singular}. If $b_q$ is the zero map, then we say $(V,q)$ is \emph{totally singular}.
By the \emph{quadratic radical of $(V,q)$} we mean the set 
$\rad(V,q)= \{ x\in \rad(V,b_q)\mid q(x)= 0\}$.
We say that $(V,q)$ is \emph{regular} if $\rad(V,q)=\{0\}$. 
For a symmetric bilinear form $(V,b)$ over $F$, the map $q:V\rightarrow F$ given by $q_b(x)=b(x,x)$ makes $(V,q_b)$ a quadratic form over $F$.
We call $(V,q_b)$ the \emph{quadratic form associated to $(V,b)$}.
 If $\kar(F)=2$ then this quadratic form is totally singular. 
 %If $\kar(F)\neq2$, then the quadratic form associated to the polar form of a quadratic form $(V,q)$ is $(V,2q)$, giving a one-to-one correspondence between quadratic forms and symmetric bilinear forms.

Let $\rho=(V,q)$ and $\rho'=(V',q')$ be quadratic forms over $F$. By an \emph{isometry of quadratic  forms $\rho\rightarrow\rho'$} we mean an isomorphism of $F$-vector spaces $f:V\rightarrow V'$ such that $q(x)=q'(f(x))$ for all $x\in V$.
%We call $\dim_F(V)$ the \emph{dimension of $\rho$} and denote it by $\dim(\rho)$.
%	For $c\in {F}^\times$ let $c\rho$ denote the quadratic form $(V,cq)$, where $(cq)(x)=c(q(x))$ for $x\in V$. We say that $\rho$ \emph{represents an element $a\in F$} if $q(x)=a$ for some $x\in V\backslash\{0\}$. 
%For a symmetric bilinear form $\varphi$ and some $c\in {F}^\times$  we say that \emph{$\varphi$ represents $c$} if the quadratic form associated to $\varphi$ represents $c$, and we denote the set of non-zero elements represented by $\varphi$ by $D(\varphi)$. 
 We say $\rho$ is \emph{isotropic} if
 $q(x)=0$ for some $x\in V\backslash\{0\}$, and \emph{anisotropic} otherwise. 
 By \emph{a totally isotropic subspace  of $\rho$} we mean an $F$-subspace $W$ of $V$ such that $q|_{W}=0$. 
 We call the maximum of the  dimensions of all totally isotropic subspaces of $\rho$ the \emph{Witt index of $\rho$},  denoted $i_W(\rho)$.
 Assume $\rho$ is nonsingular. Then $i_W(\rho)\leqslant\frac{1}{2}\dim(\rho)$ (see \cite[(7.28)]{Elman:2008}) and if $i_W(\rho)= \frac{1}{2}\dim (\rho)$ 
we say that $\rho$ is \emph{hyperbolic}.  We denote the anisotropic part of $\rho$ by $\rho_\an$ (see  \cite[(8.5)]{Elman:2008}).
 %The maximal dimension of a totally isotropic subspace of $\varphi$ is called the \emph{Witt index of $\rho$} and denoted by $\iota(\rho)$. 
 % We say that $\rho$ \emph{represents an element $a\in {F}^\times$} if there exists an $x\in V\backslash\{0\}$ such that $q(x)=a$. 
%	For $c\in {F}^\times$ let $c\rho$ denote the quadratic form $(V,cq)$, where $(cq)(x)=c(q(x))$ for $x\in V$.
%For any nonsingular quadratic form   $\rho$  over $F$, there exist nonsingular quadratic forms $\rho'$  and $\rho''$ over $F$ with $\rho'$ anisotropic and $\rho''$ hyperbolic such that $\rho\simeq\rho'\perp\rho''$.  The form  $\rho'$  is uniquely determined up to isometry and we call it \emph{the anisotropic part of $\rho$} (see  \cite[(8.5)]{Elman:2008}), and we denote it by $\rho_\an$.

%Let $\rho_1$ and $\rho_2$ be quadratic forms over $F$.
%By an \emph{isometry of quadratic forms from $\rho_1$ to $\rho_2$} we mean an isomorphism of $F$-vector spaces $f:V\longrightarrow W$ such that $q=q'\circ f $. If such an isometry exists, we say $\rho_1$ and $\rho_2$ are \emph{isometric}.
We say that the quadratic forms $\rho_1$ and $\rho_2$ over $F$ are \emph{similar} if there exists an element $c\in {F}^\times$ such that $\rho_1\simeq c\rho_2$.
%The \emph{orthogonal sum of the quadratic forms $\rho_1$ and $\rho_2$}, denoted $\rho_1\perp\rho_2$, is defined to be pair $(V\times W,q'')$ with $q'':(V\times W)\rightarrow F$ given by $q''((v, w))= q'(v) +q(w)$ for all $v\in V$ and $w\in W$.
%We say that $\rho_2$ is a \emph{subform of $\rho_1$} if there exists a quadratic form $\rho_3$ over $F$ such that $\rho_1\simeq \rho_2\perp\rho_3$. 
Recall the concept of a tensor product of a symmetric or alternating bilinear form and a quadratic form (see \cite[p.51]{Elman:2008}).
%Let $\varphi=(V,b)$ be a symmetric bilinear form over $F$. Then there is a  quadratic form $(V\otimes_F W, b\otimes q)$ over $F$  where the quadratic  map 
  %$b\otimes q:V\otimes_F W\rightarrow F$ is determined by 
%$(b\otimes q)( v\otimes w)= b(v,v)\cdot q(w) $
%for all $w\in W, v\in V$  (see \cite[p.51]{Elman:2008}). We call the quadratic form $(V\otimes_F W, b\otimes q)$ the \emph{tensor product of $\varphi$ and $\rho$}, and we denote it by $\varphi\otimes \rho$. 
For a quadratic form $\rho$ over $F$, we say that $\varphi$ \emph{factors $\rho$} if there exists a quadratic form $\rho'$ over $F$ such that $\rho\simeq \varphi\otimes \rho'$. Similarly, we say a quadratic form $\rho'$ over $F$ \emph{factors $\rho$} if there exists a symmetric bilinear form $\varphi$ over $F$ such that $\rho\simeq \varphi\otimes\rho'$.

%Let $V$ be an $F$-vector space. 
%A quadratic form $(V,q)$ corresponds to a homogeneous quadratic polynomial in $n$ variables over $F$ by identifying $V$ via an $F$-basis with $F^n$ for $n=\dim_F(V)$.
%For $a_1,a_2\in F$ we denote by $[a_1,a_2]$ the quadratic form  $(F^2,q)$ where  $q:F^2\times F^2\rightarrow F$ is given by $(x,y)\mapsto a_1x^2 + xy + a_2y^2.$ We call the quadratic form $[0,0]$ the \emph{hyperbolic plane} and denote it by $\HH$.

%Let  $\rho$ be a nonsingular quadratic form over $F$.  Then there exists nonsingular quadratic forms $\rho'$ and $\rho''$ over $F$ with $\rho'$ anisotropic and $\rho''$ hyperbolic
 %such that \linebreak $(V,q)\simeq \rho'\perp \rho''$.  In this decomposition  $\rho'$  is uniquely determined up to isometry. We  call $\rho'$ \emph{the anisotropic part of $\rho$} (see  \cite[(8.5)]{Elman:2008}), and we denote it by $\rho_\an$.
%We say that two nonsingular quadratic forms $\rho$ and $\pi$ over $F$ are \emph{Witt equivalent} if $\rho_\an\simeq \pi_\an$. 
%This is an equivalence relation on the set of nonsingular quadratic forms over $F$.% and the equivalence classes form  the \emph{Witt group of  nonsingular quadratic forms over $F$}, which we denote by $W_q(F)$. The group operation in $W_q(F)$ is induced by the orthogonal sum  (see \cite[\S8]{Elman:2008}). We denote the class of $\rho$ in $W_q(F)$ by  $[\rho]$.

For a positive  integer $m$, by an \emph{$m$-fold (quadratic) Pfister form over $F$}  we mean a quadratic form that is isometric to the tensor product of a $2$-dimensional nonsingular quadratic form representing $1$ and an $(m-1)$-fold bilinear Pfister form over $F$.
 Pfister forms are either anisotropic or hyperbolic (see \cite[(9.10)]{Elman:2008}).

% The following result is well known, but we include a proof for completeness. 

%\begin{lemma}\label{cor:roundpf} 
%Let $\pi$ and $\rho$ be quadratic forms over $F$ such that $\pi$ is a Pfister form and $\rho$ represents $1$. Then $\pi$ is similar to $\rho$ if and only if $\pi\simeq \rho$.
%\end{lemma}
%\begin{proof}
%To show the nontrivial implication, assume that $a\in\mg{F}$ is such that $\pi\simeq a\rho$. Then as $\rho$ represents $1$ we have that $\pi$ represents $a$. By   \cite[(9.9)]{Elman:2008} we have that $\pi\simeq b\cdot \pi$ for all elements represented by $\pi$, hence 
%$\rho\simeq a^2\rho\simeq  a\pi \simeq \pi\,.$ 
%\end{proof}

 %By a \emph{rational function field over $F$} we mean a field extension $K/F$ such that there exist $n\in\NN$ and elements $t_1,\dots,t_n\in K$ algebraically independent over $F$ such that $K=F(t_1,\dots,t_n)$.%; in this case $K/F$ is finitely generated of transcendence degree $n$.

Let $\rho$ be a regular quadratic form over $F$.
If $\dim(\rho)\geqslant 3$ or if $\rho$ is anisotropic and $\dim(\rho)=2$, then we call the function field of the projective quadric over $F$ given by $\rho$ the \emph{function field of $\rho$} and denote it by $F(\rho)$. In the remaining cases we set $F(\rho)=F$.
This agrees with the definition in \cite[\S22]{Elman:2008}.
%By \cite[(22.9)]{Elman:2008}
%then $F(\rho)/F$ is a rational function field over $F$ if and only if $\rho$ is isotropic over $F$.
For an anisotropic symmetric bilinear form $\varphi$, the  quadratic form $\rho$ associated to $\varphi$ is regular.  We call $F(\rho)$ the \emph{function field of $\varphi$} and we denote it by $F(\varphi)$. 
Let $K/F$ be a field extension. Then we write $(V,q)_K=(V\otimes_F K, q_K) $ where $q_K$ is  the unique quadratic map such that $q_K( v\otimes k )= k^2q(v)$ for all $v\in V$ and $k\in K$.  %This induces a natural inclusion $W_q(F)\rightarrow W_q(K)$.

\begin{prop}\label{prop:funcfieldbf}
Let $\rho$ be a nonsingular quadratic  form and let $\varphi$ be an anisotropic  bilinear Pfister form over $F$. If $\rho_{F(\varphi)}$ is hyperbolic then either $\rho$ is hyperbolic or  $\varphi$ factors $\rho_\an$. 
\end{prop}
\begin{proof}  See  \cite[(5.2)]{HoffmannLaghribi:Isoqfffquadricc2}. 
%If $\pi$ is hyperbolic, then we may choose $\rho$ to be any hyperbolic $(m-n)$-fold Pfister form over $F$. 
%If $\kar(F)\neq2$ the result follows immediately from \cite[(23.6) and (24.1)]{Elman:2008} and the one-to-one correspondence between quadratic forms and symmetric bilinear forms.
%Otherwise, by \cite[(1.4)]{Laghribi:wittkerffc2}, we have that $\pi\simeq \varphi'\otimes \rho'$ for some quadratic form $\rho'$ and  an $n$-fold bilinear Pfister form $\varphi'$ over $F$ such that 
%$\dim(\varphi)=\dim(\varphi')$ and $D(\varphi)=D(\varphi')$. 
%It follows that the  quadratic forms associated to $\varphi$ and $\varphi'$ are isometric (see \cite[(8.1)]{HoffmannLaghribi:qfpfisterneigbourc2}).
%By  \cite[(9.9)]{Elman:2008}, we have that $a\cdot\pi\simeq \pi$ for any element represented by $\pi$. In particular, as $\varphi'$ represents $1$, any element represented by $\rho$ is represented by $\pi$, and hence we can scale $\pi$ to assume that $\rho$ represents $1$. Therefore the quadratic form associated to $\varphi'$ and $\varphi$ is a subform of $\pi$ in the sense of \cite[\S 7, p.41]{Elman:2008}.  
 %By  \cite[(24.1,(2))]{Elman:2008} we then have that $\pi\simeq \varphi\otimes\rho$ for some $(m-n)$-fold Pfister form $\rho$.
\end{proof}
%
%the associated quadratic form to $\varphi'$ is isometric to the associated quadratic form to $\varphi$. By   \cite[(24.1)]{Elman:2008} we may assume that $\rho'$ is an $(m-n)$-fold Pfister form.
%
%Let $b_1,\ldots, b_m\in F^\times$ be such that $\varphi\simeq\pff{b_1,\ldots, b_m}$. Then there exists $x_1\ldots, x_m\in F$ and $y_1,\ldots, y_m\in F^\times$ such that  $\varphi'\simeq \pff{x_1^2 + y_1^2b_1,\ldots,x_m^2 + y_m^2b_m}$.  As we can write $\varphi'\otimes\rho$ as $ \qf{1}_b\otimes \varphi'\otimes\rho' $ and $\qf{1}_b$ represents all squares in $F$, it follows from  \cite[(15.6)]{Elman:2008} and the isometry $\qf{c^2 d}_b\simeq \qf{d}_b$ for all $c,d\in F^\times$ that there exists an $(m-n)$-fold Pfister form $\rho''$ such that
%$$ \pff{x_1^2 + y_1^2b_1,\ldots,x_m^2 + y_m^2b_m} \otimes\rho'\simeq \pff{b_1,x_2^2 + y_2^2b_2, \ldots,x_m^2 + y_m^2b_m} \otimes\rho''\,. $$
%Repeating this argument for the other entries in the bilinear Pfister form gives the result.

\section{Algebras with involution}\label{section:basicsalgs}

We refer to \cite{pierce:1982} as a general reference on finite-dimensional algebras over fields, and for central simple algebras in particular, and to \cite{Knus:1998} for involutions.
Let $A$ be an (associative) $F$-algebra.  We denote the centre of $A$ by $Z(A)$.
For a field extension $K/F$, the $K$-algebra $A\otimes_F K$ is denoted by $A_K$.
An element $e\in A$ is called an \emph{idempotent} if $e^2=e$.
 An $F$-\emph{involution on $A$} is an $F$-linear map $\sigma:A\rightarrow A$ such that  $\sigma(xy)=\sigma(y)\sigma(x)$ for all $x,y\in A$ and  $\sigma^2=\textrm{id}_A$. 

%In this section we recall the basic definitions and results we use on central simple algebras with involution. 

Assume now that $A$ is finite-dimensional and simple  (i.e.~it has no nontrivial two-sided ideals).
By Wedderburn's Theorem (see \cite[(1.1)]{Knus:1998}), $A\simeq\End_D(V)$ for a finite-dimensional $F$-division algebra $D$  and a finite-dimensional right $D$-vector space $V$.
Furthermore, the centre of $A$ is a field and   $\dim_{Z(A)}(A)$ is a square number, whose positive square root is called the \emph{degree of $A$} and is denoted $\mathrm{deg}(A)$. The degree of $D$ is called the \emph{index of $A$} and denoted $\mathrm{ind}(A)$. We call $A$ \emph{split} if $\mathrm{ind}(A)=1$, that is $A\simeq \End_F(V)$ for some finite-dimensional right $F$-vector space $V$. 
If $Z(A)=F$, then we call the $F$-algebra $A$ \emph{central simple} and 
we call a field extension $K/F$ such that  $A_K$ is split a \emph{splitting field of $A$}.  
%Two central simple $F$-algebras $A$ and $B$ are called \emph{Brauer equivalent} if $A$ and $B$ are isomorphic to endomorphism algebras of two right vector spaces over the same $F$-division algebra. 
If $A$ is a central simple $F$-algebra then we denote $\Trd_A:A\lra F$ the reduced trace map  and $\Nrd_A:A\lra F$ the reduced norm map, as defined in \cite[(1.6)]{Knus:1998}.

By an \emph{$F$-algebra with involution} we mean a pair $(A,\sigma)$ of a finite-dimensional central simple $F$-algebra $A$ and an $F$-involution $\sigma$ on $A$ (note that we only consider involutions that are  linear with respect to the centre of $A$, that is involutions of the first kind, here).  We use the following notation:
$\mathrm{Sym}(A,\sigma)= \{ a\in A\mid\sigma(a)=a\}$,
$\mathrm{Skew}(A,\sigma)= \{ a\in A\mid\sigma(a)=-a\}$ and
$\Alt(A,\sigma)= \{ a-\sigma( a)\,|\, a\in A\}$.
These are $F$-linear subspaces of $A$.

Let  $(A,\s)$ and $(B,\tau)$ be $F$-algebras with involution.
By an \emph{isomorphism of $F$-algebras with involution}  $\Phi:(A,\sigma)\rightarrow(B,\tau)$ we mean an $F$-algebra isomorphism $\Phi: A\rightarrow B$ satisfying $\Phi\circ\sigma=\tau\circ\Phi$.
On the $F$-algebra $A\otimes_F B$ we obtain an $F$-involution $\s\otimes \tau$, whereby $(A\otimes_F B,\sigma\otimes \tau)$ is an $F$-algebra with involution, which we denote by $(A,\sigma)\otimes(B,\tau)$.
For a field extension $K/F$ we write $(A,\s)_K=(A\otimes_F K,\s\otimes\mathrm{id})$.

We call $(A,\sigma)$ \emph{isotropic}  if  there exists $a\in A\backslash\{0\}$ such that  $\sigma(a)a=0$, and \emph{anisotropic} otherwise.  An idempotent  $e\in A$ is called  \emph{metabolic with respect to $\sigma$}  if $\sigma(e)e=0$ and $\dim_FeA=\frac{1}{2}\dim_F A$. We call $(A,\sigma)$   \emph{metabolic}  if $A$ contains a  metabolic idempotent element with respect to $\sigma$. For more information on metabolic involutions, see \cite{dolphin:metainv}.

%\begin{prop}\label{prop:metabf}
%Let $\varphi$ be a nondegenerate alternating or symmetric bilinear form over $F$. Then  $\varphi$ is isotropic (resp. metabolic) if and only if $\Ad(\varphi)$ is isotropic (resp. metabolic).
%\end{prop}
 %\begin{proof} See \cite[(4.3)]{dolphin:orthpfist} for the statement of isotropic and  \cite[(4.8)]{dolphin:metainv} for the statement on metabolicity.
% \end{proof}

To every nondegenerate symmetric or alternating bilinear form $\varphi=(V,b)$ over $F$ we can associate an algebra with involution   in the following way. Let $A=\End_F(V)$. Then there is a unique involution $\sigma$ on $A$ such that 
$$b(x,f(y))=b(\sigma(f)(x),y) \quad \textrm{ for all } x,y\in V \textrm{ and all } f\in A .$$
We denote this $F$-involution on $A$ by $\ad_b$.
We call $(A,\ad_b)$ the \emph{$F$-algebra with involution adjoint to $\varphi$} and we denote it by $\Ad(\varphi)$. For every split $F$-algebra with involution $(A,\sigma)$, there exists a nondegenerate symmetric or alternating bilinear form $\psi$ over $F$ such that $(A,\sigma)\simeq\Ad(\psi)$ (see \cite[(2.1)]{Knus:1998}). The following is well-known, but we include a proof for completeness.

\begin{prop}\label{lemma:tensorbf} Let $\varphi$ and $\psi$ be  nondegenerate symmetric   bilinear forms over $F$. Then 
$\Ad(\varphi\otimes \psi)\simeq \Ad(\varphi)\otimes \Ad(\psi). $
\end{prop} 
\begin{proof} Let $\varphi=(V,b)$ and $\psi=(W,b')$.
Let  $f\in \End_F(V)$ and $g\in \End_F(W)$. Then for all  $u,v\in V$,  $w,t\in W$ we have
\begin{eqnarray*}(b\otimes b') (f\otimes g(u\otimes w), (v\otimes t))& =&(b\otimes b')((f(u)\otimes g(w)), (v\otimes t))
\\ &= &b(f(u), v)\cdot b'(g(w), t) 
\\ &=&  b(u, \ad_b(f)(v))\cdot b'(w, \ad_{b'}(g)(t))
\\ &=& (b\otimes b') ((u\otimes w), (\ad_b(f)(v)\otimes \ad_{b'}(g)(t)))\,.
\end{eqnarray*}
Therefore, by bilinearity of $b\otimes b'$, we have that $\ad_{b\otimes b'}(f\otimes g)= \ad_b(f)\otimes \ad_b(g) $.
Using this, it follows from the linearity of $\ad_{b\otimes b'}$  that the natural isomorphism of $F$-algebras $\Phi:\End_F(V)\otimes_F\End_F(W)\rightarrow \End_{F}(V\otimes_F W)$ is   an isomorphism of the $F$-algebras with involution in the statement.
%Consider the  natural  isomorphism of $F$-algebras $\Phi:\End_F(V)\otimes_F\End_F(W)\rightarrow \End_{F}(V\otimes_F W).$
%Using this isomorphism %,   every element  in  $\End_F(V\otimes_F W)$ 
%can be written as   $f_1\otimes g_1+\ldots + f_n\otimes g_n$
%for some for some $n\in\mathbb{N}$,  $f_1,\ldots, f_n\in \End_F(V)$ and $g_1,\ldots, g_n\in \End_F(W)$.
%Therefore, as 
%and that $\ad_{b\otimes b'}(f\otimes g)= \ad_b(f)\otimes \ad_b(g) $ for all $f\in \End_F(V)$ and $g\in \End_F(W)$, it follows  by the linearity of $\ad_{b\otimes b'}$ that  $\Phi$ is an isomorphism of $F$-algebras with involution.
\end{proof}

We distinguish two \emph{types} of $F$-algebras with involution. An $F$-algebra with involution  is said to be \emph{symplectic} if it becomes adjoint to an alternating bilinear form over some splitting field of the associated $F$-algebra, and \emph{orthogonal} otherwise.  
In characteristic different from two, these types are distinguished by the dimensions of the spaces of symmetric and alternating elements. However in characteristic two these dimensions do not depend on the type (see \cite[(2.6)]{Knus:1998}).

An $F$-\emph{quaternion algebra} is a central simple $F$-algebra of degree $2$. 
 Let $Q$ be an $F$-quaternion algebra.
By  \cite[(2.21)]{Knus:1998}, the map $Q\rightarrow Q,$  $x\mapsto \Trd_Q(x)-x$ is the unique symplectic involution on $Q$; it is called the \emph{canonical involution of $Q$}.
Any $F$-quaternion algebra has a basis $(1,u,v,w)$ such that
%\begin{eqnarray}\label{eqnarray:qatbas}u^2 =u+a, v^2=b\,\textrm{ and }\, w=uv=v-vu\,,\end{eqnarray}
$$u^2 =u+a, v^2=b\,\textrm{ and }\, w=uv=v-vu\,,$$
for some  $a\in F$ with $-4a\neq 1$ and $b\in F^\times$
 (see  \cite[Chap.~IX, Thm.~26]{Albert:1968});  such a basis is called an \emph{$F$-quaternion basis}.
 Conversely, for  $a\in F$ with $-4a\neq 1$ and $b\in F^\times$
 the above relations uniquely determine an $F$-quaternion algebra (up to $F$-isomorphism), which we denote by $[a,b)$. 
By the above, up to isomorphism any $F$-quaternion algebra is of this form. The following result can be recovered from \cite[p.104, Thm.~4]{Draxl:1983}, but we include a direct argument. 

\begin{lemma}\label{qrep}
Let $Q$ be an $F$-quaternion algebra and  $v\in Q\setminus F$ be such that $v^2\in F^\times$. There exist an element and $u\in Q$ such that $uv=v(1-u)$ and $u^2-u=a$ for some $a\in F$ with $-4a\neq 1$. That is, $(1,u,v,uv)$ is an $F$-quaternion basis of $Q$.
\end{lemma}
\begin{proof}
If $\kar(F)\neq 2$ then
it is well-known that
there exists an invertible element $x\in Q$ such that $xv+vx=0$ and we set $u=x+\frac{1}2$.
Assume that $\kar(F)=2$.
Consider the $F$-linear map $Q\to Q, x\mapsto xv+vx$. Its kernel and its image are equal to $F[v]$.
Hence  $uv+vu=v$ for some $u\in Q$, and it follows that $u^2+u\in F[v]\cap F[u]=F$.
\end{proof}

%%%%%%%%%%%%%%%%%%%%%%%%%%%%%%%%%%%%%%

 With an $F$-quaternion basis $(1,u,v,w)$ of $Q$,
 we define $F$-involutions $\gamma$, $\s$ and $\tau$ on $Q$ via their action on $u$ and $v$ as follows. We let 
 \begin{eqnarray*}
 \gamma: & u\mapsto 1-u, &  v\mapsto-v\\
  \s: & u\mapsto 1-u,  &v\mapsto \hphantom{-}v\\
   \tau:& u\mapsto  u,\hphantom{-1} & v\mapsto -v\,. 
   \end{eqnarray*}
Note that $\gamma$ is the canonical involution on $Q$. Further, if $\kar(F)=2$, then $\gamma=\s$ and hence $\s$ is symplectic. Otherwise $\s$ is orthogonal.   We use the notation $$[a\qilr b)=(Q,\gamma), \quad  [a\qil b)=(Q,\s),\quad[a\qir b)= (Q,\tau)\,.$$

\begin{prop}\label{prop:orth}
Let $(Q,\s)$ be an $F$-quaternion algebra with orthogonal involution. Then there exists $a\in F$ with $-4a\neq 1$ and $b\in F^\times$ such that $(Q,\s)\simeq [a\qir b)$.
\end{prop}
\begin{proof} 
Let $\gamma$ be the canonical involution of $Q$.
By \cite[(2.21)]{Knus:1998}, we have $\s=\Int(v)\circ \gamma$ for some invertible  element $v\in \mathrm{Skew}
(Q,\gamma)\backslash F$. 
Then
$v^2=-\gamma(v)v \in F^\times$ and we set $b=v^2$. By \cref{qrep}, there exists an element $u\in Q$ such
that  $uv=v(1-u)$  and $u^2-u=a$ for some $a\in F$ with $-4a\neq 1$. Hence $(1,u,v,uv)$ is an $F$-quaternion basis of $Q$ and we have that $\gamma(u)=1-u$ and $\gamma(v)=-v$. Hence $\s(u)=u$ and $\s(v)=-v$, and further $(Q,\s)\simeq [a\qir b)$.
\end{proof}

We call an $F$-algebra with involution \emph{totally decomposable} if it is isomorphic to a tensor product of $F$-quaternion algebras with involution. Note that by \cref{lemma:tensorbf}, the  $F$-algebra with involution adjoint to a bilinear Pfister form over $F$ is totally decomposable.

Let $(A,\sigma)$ be an $F$-algebra with orthogonal involution. 
By \cite[(7.1)]{Knus:1998}, for any $F$-algebra with orthogonal involution $(A,\s)$ with $\deg(A)$ even and  any $a,b\in \Alt(A,\s)$ we have $\Nrd_A(a)F^{\times 2}=\Nrd_A(b) F^{\times 2}$.
 Therefore, as in \cite[\S7]{Knus:1998}, we may make the following definition.  The \emph{determinant of $(A,\sigma)$}, denoted $\Delta(A,\sigma)$, is the square class of the reduced norm of an arbitrary  alternating unit, that is
$$\Delta(A,\sigma)= \mathrm{Nrd}_A(a)\cdot F^{\times2}\textrm{ in } F^\times/ F^{\times2} \quad \textrm{ for any} \,\,a\in \Alt(A,\sigma)\cap A^\times.$$

For the rest of  this section, we assume that $\kar(F)=2$.
Let $(A,\sigma)$ be a totally decomposable   $F$-algebra with orthogonal involution. That is  $$(A,\sigma)\simeq \bigotimes_{i=1}^n(Q_i,\sigma_i)$$ where $(Q_i,\sigma_i)$  are   $F$-algebras with involution for $i=1,\ldots, n$. 
Note that we must have that $(Q_i,\s_i)$ is orthogonal for all $i=1,\ldots, n$ by 
\cite[(2.23)]{Knus:1998}.
Let $d_i=\Delta(Q_i,\sigma_i)$. Then the bilinear Pfister form $\pi=\pff{d_1,\ldots, d_n}$ over $F$ does not depend on the choice of the $F$-quaternion algebras with involution $(Q_i,\s_i)$ in the decomposition of $(A,\s)$ by \cite[(7.3)]{dolphin:orthpfist}. We call this bilinear Pfister form \emph{the Pfister invariant of $(A,\sigma)$} and denote it by $\mathfrak{Pf}(A,\sigma)$. 
Note that by \cite[(7.3)]{dolphin:orthpfist}, for any field extension $K/F$ we have that $\mathfrak{Pf}((A,\s)_K)= (\mathfrak{Pf}(A,\s))_K$.

\begin{prop}\label{thm:pfisterinvar} Assume $\kar(F)=2$.
Let $(A,\sigma)$ be a totally decomposable    $F$--algebra with orthogonal involution. Then
$(A,\sigma)$ is anisotropic (resp.~metabolic) if and only if $\mathfrak{Pf}(A,\sigma)$ is anisotropic (resp.~metabolic).
\end{prop}
\begin{proof}
See \cite[(7.5)]{dolphin:orthpfist}.
\end{proof}

\section{Algebras with quadratic pair}

 We now recall the definition of and basic results we use on quadratic pairs.
% In characteristic different from $2$, a quadratic pair is an equivalent concept to an orthogonal involution (see \cite[p.56]{Knus:1998}). 
Let $(A,\sigma)$ be an $F$-algebra with involution.
We call an $F$-linear map $f:\Sym(A,\sigma)\rightarrow F$ a \emph{semi-trace on $(A,\sigma)$} if it satisfies $f(x+\sigma(x)) = \mathrm{Trd}_A(x)$  for all $x\in A$. By \cite[(4.3)]{dolphin:quadpairs}, if $\kar(F)\neq2$, then $\frac{1}{2}\Trd_A|_{\Sym(A,\sigma)}$ is the unique semi-trace on $(A,\sigma)$. On the other hand, if $\kar(F)=2$, then the existence of a semi-trace on $(A,\sigma)$ implies that $\Trd_A(\Sym(A,\sigma))=\{0\}$ and hence by \cite[(2.6)]{Knus:1998} that $(A,\sigma)$  is symplectic.  

Given an element $\ell\in A$ with $\ell+\s(\ell)=1$, the map $f:\Sym(A,\sigma)\rightarrow F$  given by  $x\mapsto \Trd_A(\ell x)$  is a semi-trace on $(A,\sigma)$, and conversely every semi-trace on $(A,\sigma)$ is of this form by \cite[(5.7)]{Knus:1998} (although  the case where $\kar(F)\neq 2$ and $(A,\s)$ is symplectic is excluded there, the same proof applies). In this case, we say that the semi-trace $f$ on $(A,\s)$ \emph{is given by $\ell$}.
% or just \emph{the semi-trace given by $\ell$} when $(A,\s)$ is clear from the context.
For another element $\ell'\in A$ such that $\ell'+\s(\ell')=1$, we have that 
$\ell$ and $\ell'$ give the same semi-trace on $(A,\s)$ if and only if $\ell-\ell'\in\Alt(A,\s)$ (see \cite[(5.7)]{Knus:1998}).

An \emph{$F$-algebra with quadratic pair} is a triple $(A,\sigma,f)$ where $(A,\sigma)$ is an $F$-algebra with involution, which is assumed to be orthogonal if $\kar(F)\neq 2$ and symplectic if $\kar(F)=2$, and where $f$ is a semi-trace on $(A,\sigma)$.
In  $\kar(F)\neq2$ the concept of an algebra with quadratic pair is  equivalent to the concept of an algebra with orthogonal involution, as 
 then the semi-trace given by $\frac{1}{2}$  is the  unique semi-trace on $(A,\sigma)$.

Given two $F$-algebras with quadratic pair $(A,\sigma,f)$ and $(B,\tau,g)$, by an \emph{isomorphism of  $F$-algebras with quadratic pair} $\Phi:(A,\sigma,f)\rightarrow(B,\tau,g)$ we mean an isomorphism of the underlying $F$-algebras with involution satisfying $ f=g\circ\Phi$.

Let $(A,\sigma,f)$ be an $F$-algebra with quadratic pair.  We call $(A,\sigma,f)$ \emph{isotropic} if there exists an element $s\in\Sym(A,\sigma)\backslash\{0\}$ such that $s^2=0$ and $f(s)=0$, and \emph{anisotropic} otherwise.  
In particular,  if $(A,\s,f)$ is isotropic, then $A$  has zero divisors.
We call an idempotent $e\in A$ \emph{hyperbolic with respect to $\sigma$ and $f$} if $\sigma(e)=1-e$ and $f(eA\cap \Sym(A,\sigma))=\{0\}$.
We say that the $F$-algebra with quadratic pair $(A,\sigma,f)$ is \emph{hyperbolic} if $A$ contains a hyperbolic  idempotent with respect to $\sigma$ and $f$.

 We describe, 
following \cite[\S5]{Knus:1998},
 how a nonsingular quadratic form gives rise to an algebra with quadratic pair.
Let $\rho=(V,q)$ be a nonsingular quadratic form over $F$ with polar form $(V,b_q)$. 
By declaring
$$(v_1\otimes w_1)\ast(v_2\otimes w_2)=b_q(w_1,v_2)\cdot(v_1\otimes w_2)\quad \textrm{ for }\quad \,v_1,v_2,w_1,w_2\in V\,$$
a product $\ast$ is defined on the tensor product $V\otimes_F V$ making it into an $F$-algebra.
By  declaring $\sigma(v\otimes w)=w\otimes v$ for $v,w\in V$ we obtain an $F$-involution $\s$ on the $F$-algebra $V\otimes_FV$. 
Then by \cite[(5.1)]{Knus:1998}, the $F$-linear map $\Phi:V\otimes_F V\rightarrow \End_F(V)$ determined by $$\Phi(u\otimes v)(w)= b_q(v,w) u\quad \mbox{ for } \, u,v,w\in V$$ yields an isomorphism of $F$-algebras with involution $\Ad(V,b_q)\lra (V\otimes_F V,\s)$.
% where $\s$ is the `switch map' $V\otimes V\lra V\otimes V$, given by $\sigma(v\otimes w)=w\otimes v$ for $v,w\in V$.
According to \cite[(5.11)]{Knus:1998} there is a unique semi-trace $f_q:\Sym(\Ad(V,b_q))\rightarrow F$  such that 
$f_q(\Phi(v\otimes v))=q(v)$  for   $v\in V,$ which yields an $F$-algebra with quadratic pair $$\Ad(\rho)=(\End_F(V),\ad_{b_q}, f_q)\,,$$ called the \emph{adjoint $F$-algebra with quadratic pair  of $\rho$}.
We say that an $F$-algebra with quadratic pair $(A,\sigma,f)$ is \emph{adjoint to  $\rho$} if $(A,\sigma,f)\simeq \Ad(\rho)$. 
By \cite[(5.11)]{Knus:1998}, for any split $F$-algebra with quadratic pair $(A,\s,f)$, there exists a nonsingular quadratic form $\rho$ over $F$ such that $(A,\sigma,f)\simeq\Ad(\rho)$, and for two nonsingular quadratic forms $\rho$ and $\rho'$ over $F$, we have that $\Ad(\rho)\simeq\Ad(\rho')$ if and only if $\rho$ and $\rho'$ are similar.

%\begin{prop}\label{prop:similar}
%Let $\rho_1$ and $\rho_2$ be quadratic forms over $F$. Then  $\Ad(\rho_1)\simeq\Ad( \rho_2)$ if and only if $\rho_1$ is similar to $\rho_2$. 
%\end{prop}
%\begin{proof}
%If $\kar(F)=2$ this is shown in \cite[(1.5)]{tignol:qfskewfield}. If $\kar(F)\neq 2$ then $\rho_1$ and $\rho_2$ are determined by their polar forms and the result follows from  \cite[p.1]{Knus:1998}. 
%\end{proof}

 \begin{prop}\label{prop:hypiffquad}
Let $\rho$ be a nonsingular quadratic form over $F$. Then $\rho$ is isotropic (resp.~hyperbolic) if and only if $\Ad(\rho)$ is  isotropic (resp.~hyperbolic).
\end{prop}
\begin{proof} 
The statement on isotropy follows from \cite[(6.3) and (6.6)]{Knus:1998}. See  \cite[(6.13)]{Knus:1998} for the statement on hyperbolicity.
\end{proof}

 For any field extension $K/F$ we will use the notation  $(A,\sigma,f)_K$ for the $K$-algebra with quadratic pair $(A_K,\sigma_K,f_K)$ where
$f_K:\mathrm{Sym}(A_K,\s_K)\lra K$ is the canonical extension of $f$ to a 
  $K$-linear map.

\section{Tensor products of involutions and  quadratic pairs}\label{tenprodqps}

In this section we consider the tensor product of an algebra with involution with an algebra with quadratic pair. This corresponds  to notion of the tensor product of a symmetric bilinear form and a quadratic form.
We show that this tensor product is  associative with the tensor product of algebras with involution. This property underlies our definition of  a totally decomposable quadratic pair in the following section.

\begin{prop}\label{lemma:explict}
Let $(A,\s,f)$ be an $F$-algebra with quadratic pair  and $(B,\tau)$ an $F$-algebra with involution.
Then there is a unique semi-trace $g$ on $(B,\tau)\otimes (A,\s)$ such that $g(s_1\otimes s_2) =  \Trd_B(s_1)\cdot f(s_2)$
for all  $s_1\in \Sym(B,\tau)$ and  $s_2\in\Sym(A,\s)$.
Moreover, if the semi-trace $f$ on $(A,\s)$ is given by $\ell$, then 
 $g$ is the semi-trace on $(B,\tau)\otimes(A,\s)$ given by $1\otimes \ell$.
\end{prop}
\begin{proof}  If $\kar(F)\neq 2$ then this result is trivial. Assume that $\kar(F)=2$. 
Note   that as $1\otimes \ell  + (\tau\otimes\s)(1\otimes \ell) = 1\otimes 1$, the element $1\otimes  \ell$ gives a semi-trace on $(B\otimes_F A,\tau\otimes \s)$ by  \cite[(5.7)]{Knus:1998}. 
For all $s_1\in\Sym(B,\tau)$ and $s_2\in\Sym(A,\s)$ we have that 
$$\Trd_{B\otimes_F A}((1\otimes\ell)(s_1\otimes s_2))= \Trd_B(s_1)\cdot \Trd_A(\ell \cdot s_2)=\Trd_B(s_1)\cdot f(s_2)\,.$$
For the uniqueness statement, see \cite[(5.18)]{Knus:1998}.
\end{proof}

Let $(A,\s,f)$ be an $F$-algebra with quadratic pair and let $(B,\tau)$ be an $F$-algebra with involution, which is  assumed to be  orthogonal if $\kar(F)\neq 2$.  Then by \cite[(2.23)]{Knus:1998}, $(B,\tau)\otimes (A,\s)$  is orthogonal if $\kar(F)\neq 2$ and symplectic if $\kar(F)=2$.
We denote by $(B,\tau)\otimes (A,\s,f)$ the $F$-algebra with quadratic pair $( B\otimes_F A, \tau\otimes \s,g)$, where $g$ is the semi-trace  $g$ on $(B,\tau)\otimes (A,\s)$ characterised in \cref{lemma:explict}.

\begin{prop}\label{prop:tensor} Let $\varphi$ be a symmetric bilinear form over $F$ and $\rho$  a nonsingular quadratic form over~$F$. Then $\Ad(\varphi\otimes \rho)\simeq \Ad(\varphi)\otimes\Ad(\rho)$.
\end{prop}
\begin{proof} See \cite[(5.19)]{Knus:1998}.
\end{proof}
%%%%
%\begin{prop}\label{lemma:metahyp}
%Let $(A,\sigma,f)$ be an $F$--algebra with quadratic pair and  $(C,\tau)$  an $F$--algebra with involution. Assume that $(C,\tau)$ is orthogonal if $\kar(F)\neq 2$. If $(C,\tau)$ is metabolic, then $(C,\tau)\otimes (A,\sigma,f)$ is hyperbolic.
%\end{prop}
%\begin{proof} 
%See \cite[(A.5)]{Tignol:galcohomgps}.
%\end{proof}

%\begin{lemma}\label{lemma:explict}
%Let  $(B,\tau)$ be an $F$-algebra with involution  assumed orthogonal if $\kar(F)\neq 2$ and let 
%$(A,\s,f)$ be an $F$-algebra with quadratic pair  where $f=\str(A,{\sigma, \ell})$ for some $\ell\in A$ such that $\ell+\s(\ell)=1$. Then 
%$$(B,\tau)\otimes (A,\s,f)= (B\otimes_F A,\tau\otimes \s ,g)\textrm{ where } g= \str(1\otimes\ell)\,.$$
%\end{lemma}
%\begin{proof}\end{proof}

\begin{prop}\label{lemma:decomp}
 Let  $(B,\tau)$ and $(C,\gamma)$ be  $F$-algebras with involution  that are  assumed to be orthogonal if $\kar(F)\neq 2$ and let $(A,\s,f)$ be an $F$-algebra with quadratic pair. Then
$$ ((B,\tau)\otimes (C,\gamma))\otimes (A,\s,f) \simeq(B,\tau)\otimes ((C,\gamma)\otimes (A,\s,f))\,. $$
 \end{prop}
 \begin{proof} 
 Let $\Phi:(B\otimes_FC)\otimes_F A\rightarrow B\otimes_F(C\otimes_FA)$ be the natural  $F$-algebra  isomorphism. 
 Clearly $\Phi$ is compatible  with the involutions in the statement. 
 %It remains to show that $\Phi$ is compatible with the semi-traces.
 By \cite[(5.7)]{Knus:1998}, $f$ is given by some $\ell\in A$ with $\ell+\s(\ell)=1$. 
 It follows from \cref{lemma:explict} that the semi-trace associated with
  $((B,\tau)\otimes (C,\gamma))\otimes (A,\s,f) $ is given by $(1\otimes 1)\otimes\ell$ and the semi-trace associated with $(B,\tau)\otimes ((C,\gamma)\otimes (A,\s,f))$ is given by 
 $1\otimes (1\otimes\ell)$.  It then easily follows that $\Phi$ is an isomorphism between the $F$-algebras with quadratic pair.
 \end{proof}

 By \cref{lemma:decomp}, the tensor product of two algebras with involution on the one hand, and the  tensor product of an algebra with involution with an algebra with quadratic pair on the other hand, are mutually associative. That is, for $F$-algebras with involution $(A,\s)$ and $(B,\tau)$ and an $F$-algebra with quadratic pair $(C,\gamma,f)$, the expression  $(A,\s)\otimes(B,\tau)\otimes(C,\gamma,f)$ is unambiguous.

\begin{prop}\label{choicef} 
Assume $\kar(F)=2$.
Let $(B,\tau)$ and  $(C,\s)$ be  $F$-algebras with symplectic involution. Then  there exists a unique semi-trace $h$ on $(B,\tau)\otimes(C,\s)$ such that $h(s_1\otimes s_2)=0$ for all $s_1\in \Sym(B,\tau)$ and $s_2\in \Sym(C,\s)$. Moreover, for any  semi-trace $f$  on $(C,\s)$, $h$ is the
 semi-trace associated with the $F$-algebra with quadratic pair $(B,\tau)\otimes(C,\s, f)$.
\end{prop}
\begin{proof} For the existence and uniqueness of the semi-trace $h$, see \cite[(5.20)]{Knus:1998}.
Let $f$ be any semi-trace on $(C,\s)$.
 By  \cite[(5.7)]{Knus:1998}, $f$ is given by an element  $\ell\in C$ with $\ell+\s(\ell)=1$. 
 Let $g$ be the  semi-trace associated with  $(B,\tau)\otimes(C,\s,f)$. Then by 
  \cref{lemma:explict},  $g$ is
given by $1\otimes\ell$.
By \cite[(2.6)]{Knus:1998}, as $(B,\tau)$ is symplectic we have  $\Trd_B(\Sym(B,\tau))=\{0\}$. Hence for all $s_1\in \Sym(B,\tau)$ and $s_2\in \Sym(C,\s)$ we have  
$$ \Trd_{B\otimes_F C}((1\otimes \ell)(s_1\otimes s_2))= \Trd_B(s_1)\cdot \Trd_C(\ell\cdot s_2)=0\,.$$
That is, $g$ satisfies the characterising property of $h$ in the statement. Therefore by the uniqueness of the semi-trace $h$, we have that $g=h$. 
\end{proof}

Given two $F$-algebras with symplectic involution $(B,\tau)$ and $(C,\s)$, we may  define a semi-trace $h$ on $(B,\tau)\otimes (C,\s)$ in the following manner. If $\kar(F)\neq2$, then  $(B,\tau)\otimes(C,\s)$ is orthogonal by \cite[(2.23)]{Knus:1998} and we 
  let $h=\frac{1}{2}\Trd_{B\otimes_F C}$. If $\kar(F)=2$, let $h$ be the semi-trace on $(B,\tau)\otimes (C,\s)$ characterised in \cref{choicef}. We denote the $F$-algebra with quadratic pair $(B\otimes_FC,\tau\otimes\s,h)$ by $(B,\tau)\boxtimes (C,\s)$.
 If $\kar(F)=2$, then by \cref{choicef} we have that $(B,\tau)\boxtimes (C,\s)\simeq  (B,\tau)\otimes (C,\s,f)$ for any choice of semi-trace $f$ on $(C,\s)$. In particular, by Proposition~\ref{lemma:decomp}, for an $F$-algebra with symplectic involution $(A,\sigma)$, the expression $(A,\sigma)\otimes(B,\tau)\boxtimes (C,\gamma)$ is unambiguous. 
 
  Hence given a tensor product  of two $F$-algebras with symplectic involution, 
 there is natural choice of a semi-trace making this product into a quadratic pair.  We now consider this quadratic pair in the case where the $F$-algebras with involution are $F$-quaternion algebras with their canonical involutions.
Let $a\in F$ with $-4a\neq 1$ and $b\in F^\times$ and let  $(Q,\gamma) =[a\qil b)$. Recall that this $F$-algebra with involution is orthogonal if $\kar(F)\neq2$ and symplectic if $\kar(F)=2$. 
Let $u\in Q$ be such that $u^2=u+a$ and $\gamma(u)=1-u$ and let $f$ be the semi-trace on $(Q,\gamma)$ given  by $u$. Then we denote the $F$-algebra with quadratic pair $(Q,\gamma,f)$  by $\qp{a}{b}$.

\begin{prop}\label{symplectize} 
 Let $a,c\in F$ such that $4a\neq -1\neq 4c$ and $b,d\in F^\times$.
Then $${[a\qilr b)}\boxtimes[{c}\qilr{d})\simeq {[a+c+4ac \qir b)}\otimes \qp{c}{bd}\,.$$
\end{prop}
\begin{proof}
Let $(B,\sigma,f)={[a\qilr b)}\boxtimes[{c}\qilr{d})$, $(Q_1,\gamma_1)= [a\qilr b)$ and  $(Q_2,\gamma_2)=[{c}\qilr{d})$. 
Let $i,j\in Q_1$  be such that  $i^2=i+a$, $j^2= b$ and $ij=j-ji$ and let  $ u, v\in Q_2$ be such that $u^2 =u+c$, $v^2=d$ and $uv=v-vu$.  In $B$ we have that $\sigma(i\otimes 1)=1\otimes 1 - i\otimes 1$, $\sigma(j\otimes 1)=-j\otimes 1$, $\sigma(1\otimes u)=1\otimes 1-1\otimes u$ and $\sigma(1\otimes v)=-1\otimes v$.

Let $i'= i\otimes 1 + (1-2i)\otimes u$, $j'= j\otimes 1$, $u'=1\otimes u$ and $v'= j\otimes v$. Then one easily checks that $$Q'_1=F\oplus Fi'\oplus Fj'\oplus Fi'j' \quad\mbox{ and }\quad Q'_2=F\oplus Fu'\oplus Fv'\oplus Fu'v'$$ are $\s$-invariant  $F$-subalgebras of $B$ that commute elementwise with one another. 
We set $\tau_1=\s|_{Q'_1}$ and $\tau_2=\s|_{Q'_2}$. We have 
$$(Q'_1,\tau_1)\simeq [a+c+4ac \qir b) \textrm{ and } (Q'_2,\tau_2) \simeq[{c}\qil{bd})\,.$$
Hence $(B,\s)\simeq  {[a+c+4ac \qir b)}\otimes [{c}\qil{bd})$.
If $\kar(F)\neq 2$, then the semi-trace on $(B,\s)$ is uniquely determined, and in this case there is nothing further to show. 

Assume $\kar(F)=2$. 
Then $(B,\s,f)\simeq (Q_1,\gamma_1)\otimes(Q_2,\gamma_2,h)$ for any choice of semi-trace $h$ on $(Q_2,\gamma_2)$ by \cref{choicef}. Let  $h$ to be the semi-trace given by $u$.
Then for all  $s\in \Sym(Q_1\otimes_FQ_2, \gamma_1\otimes\gamma_2)=\Sym(Q'_1\otimes_FQ'_2, \tau_1\otimes\tau_2)$ we have that 
$$\Trd_{Q_1\otimes_FQ_2}((1\otimes u)\cdot s)= \Trd_{Q_1\otimes_FQ_2}(u'\cdot s)=\Trd_{Q'_1\otimes_FQ'_2}((1\otimes u')\cdot s)\,.$$ 
Hence
$(Q_1,\gamma_1)\otimes(Q_2,\gamma_2,h)\simeq  (Q'_1,\tau_1)\otimes (Q'_2,\tau_2,g) $ for the  semi-trace $g$ on $(Q'_2,\tau_2)$ given by  ${u'}$  by \cref{lemma:explict}. That is, $(B,\s,f)\simeq   {[a+c+4ac \qir b)}\otimes \qp{c}{bd}$.
\end{proof}

 \section{Totally decomposable quadratic pairs}

We call an $F$-algebra with quadratic pair \emph{totally decomposable} if it is isomorphic to a tensor product of a totally decomposable $F$-algebra with involution and an $F$-quaternion algebra with quadratic pair.  
It follows from \cref{lemma:decomp} that taking the tensor product of a totally decomposable $F$-algebra with involution and a totally decomposable $F$-algebra with  quadratic pair gives a totally decomposable  algebra with quadratic pair. 

%If $(A,\s,f)$ is a totally decomposable $F$-algebra with quadratic pair, then it follows that $(A,\s)$  is a totally decomposable  $F$-algebra with involution. % However, not every semi-trace on a totally decomposable algebra with  involution of the correct type yields  a totally decomposable algebra with quadratic pair. 
%In characteristic two, we obtain an algebra with quadratic pair by endowing an algebra with symplectic involution with a semi-trace.
%We will show  in \cref{thm:pfisterfactquad} that totally decomposable quadratic pairs on split algebras are always adjoint to a Pfister form, and hence those quadratic pairs adjoint to  quadratic forms that are not similar to Pfister forms are not totally decomposable. 
%%Conversely,  it follows easily from  \cite[(12.35)]{Knus:1998} that every  symplectic  involution  on a split algebra, and hence any symplectic involution in the triple of an adjoint quadratic pair of a quadratic form in characteristic two is a totally decomposable involution. 

Let $(A,\s,f)$ be a totally decomposable $F$-algebra with quadratic pair. Then there exists a totally decomposable $F$-algebra with quadratic pair $(B,\tau)$ and an $F$-quaternion  algebra with involution $(Q,\gamma,g)$ such that $(A,\s,f)\simeq (B,\tau)\otimes(Q,\gamma,h)$. 
If $\kar(F)\neq 2$ then $(B,\tau)$ is necessarily orthogonal. We now show that, even if $\kar(F)=2$, we may always find a decomposition as above where $(B,\tau)$ is orthogonal.  This will allow us in the next section to use the Pfister invariant of $(B,\tau)$ to study the quadratic pair $(A,\s,f)$.

\begin{prop}\label{totdecompqp} 
Let $(A,\s,f)$ be a totally decomposable $F$-algebra with quadratic pair. Then there exists a totally decomposable $F$-algebra with orthogonal involution $(B,\tau)$ and 
  an $F$-quaternion algebra with quadratic pair $(Q,\gamma,g)$ such that $(A,\s,f)\simeq  (B,\tau)\otimes (Q,\gamma, g)$.
\end{prop}
\begin{proof}  The result is trivial if $\kar(F)\neq2$. Assume that $\kar(F)=2$. 
As $(A,\s,f)$ is totally decomposable, there exist $F$-quaternion algebras with involution $(Q_i,\sigma_i)$ for $i=1,\ldots, n-1$ and an $F$-quaternion algebra with quadratic pair $(Q_n,\gamma,h)$ such that 
$$(A,\s,f)\simeq (Q_1,\s_1)\otimes \ldots\otimes  (Q_{n-1},\s_{n-1})\otimes (Q_n,\gamma, h)\,.$$
%Assume  $\kar(F)\neq 2$. Then we may consider $(A,\s,f)$ and $(Q_n,\gamma,h)$ as  $F$-algebras with orthogonal involution. As $(A,\s)$ is orthogonal, an even number of the $F$-algebras with involution $(Q_1,\s_1), \ldots, (Q_n,\gamma)$ must be even by \cite[(2.23)]{Knus:1998}. It follows from (\cref{symplectize}), ignoring the extra structure of the semi-traces in the quadratic pairs, that the  tensor product of any two $F$-algebras with symplectic involution is isomorphic to the tensor product of two $F$-algebras with orthogonal involution. The result then follows easily.
Suppose $\s_i$ is symplectic. In particular, it is the canonical involution on $Q_i$, for some $i\in\{1,\ldots, n-1\}$. 
% By (\cref{prop:orth}), there exists $a\in F$ with $-4a\neq 1$ and $b\in F^\times$ such that $(Q,\s)\simeq [a\qir b)$. Further, by (\cref{prop:qpquat}), there exists  $c\in F$ with $-4c\neq 1$ and $d\in F^\times$ such that $(Q,\gamma,h)\simeq \qp{a}{b}$.
Then by \cref{choicef}, we have that 
$$ (Q_i,\s_i)\otimes (Q_n,\gamma,h)\simeq (Q_i,\s_i)\boxtimes (Q_n,\gamma)\,. $$
 Hence,  by \cref{symplectize}, 
there exists 
 an $F$-quaternion algebra with orthogonal involution  $(Q'_i,\tau)$ and an $F$-quaternion algebra with quadratic pair  $(Q'_n,\gamma',h')$
 such that
$$ (Q_i,\s_i)\otimes (Q_n,\gamma,h)\simeq  (Q'_i,\tau)\otimes (Q'_n,\gamma',h')  \,.$$ 
Using this argument repeatedly for all $i=1,\ldots, n-1$ such that $\s_i$ is symplectic, we modify our expression of $(A,\s,f)$ above to obtain the result.
\end{proof}

%One can use the   the isomorphism from  \cref{symplectize},  similarly to how it is used in  in the proof of \cref{totdecompqp} but in the opposite direction, to give the following symplectic analog  to \cref{totdecompqp}. As we do not use this result here, we do not give details of the proof.

For interest, we also record a characteristic two specific counterpart of the previous statement, 
which produces a symplectic instead of an orthogonal factor.
\begin{prop}
Assume that $\kar(F)=2$. 
Let $(A,\s,f)$ be a totally decomposable $F$-algebra with quadratic pair with $\deg(A)=2^n$, where $n\geq 2$. Then there exist $F$-quaternion algebras with canonical  involution $(Q_i,\gamma_i)$ for $i=1,\ldots, n$  such that
$(A,\s,f)\simeq  (Q_1,\gamma_1)\otimes\ldots\otimes(Q_{n-1},\gamma_{n-1})\boxtimes (Q_n,\gamma_n)\,.$
\end{prop}
\begin{proof} 
 By  \cref{prop:orth}, for every $F$-quaternion algebra with orthogonal involution $(Q,\tau)$,  there exists an $a\in F$ and $b\in F^\times$ such that $(Q,\tau)\simeq [a\qir b)$. Similarly, by \cite[(5.6)]{dolphin:quadpairs}, for every $F$-quaternion algebra with quadratic pair $(Q,\gamma,f)$ 
  there exists an $c\in F$ and $d\in F^\times$ such that $(Q,\gamma,f)\simeq\qp{c}{d}$. The result thus follows using the isomorphism in  \cref{symplectize} in a similar way as to how it is used in  \cref{totdecompqp}, but in the opposite direction.
\end{proof}

%
%\begin{remark}
%Assume $\kar(F)=2$. 
%Let $(A,\s,f)$ be a totally decomposable $F$-algebra with quadratic pair and let $(B,\tau)$ be  a totally decomposable $F$-algebra with involution and $(Q,\gamma,h)$ an $F$-quaternion  algebra with involution  such that $(A,\s,f)\simeq (B,\tau)\otimes(Q,\gamma,h)$. In (\cref{totdecompqp}) we use the isomorphism from  (\cref{symplectize}) to show that $(B,\tau)$ can always be assumed to be orthogonal. However, we may also use the isomorphism in the other direction, similarly to how it is used in  in the proof of (\cref{totdecompqp}),  to show that $(B,\tau)$ may also always be assumed to be symplectic, and further that $(B,\tau)$ may be assumed to decompose into a product of $F$-quaternion algebras with  symplectic involution.
%\end{remark}

\section{Totally decomposable quadratic pairs on a split algebra}\label{section:main}

We  now prove our main result, that over fields of characteristic two a split algebra with totally decomposable  quadratic pair is adjoint to a Pfister form.  We use the following result, which is an approach unique to  fields of characteristic two. This approach gives more information on the $m$-fold Pfister form $\pi$ adjoint to a totally decomposable quadratic pair on an algebra  of degree $2^m$ over a field $F$ of characteristic $2$ after extending to splitting field $K$.  Specifically, we show that we can always find an  $(m-1)$-fold bilinear  Pfister form $\varphi$ defined over $F$ such  that $\varphi_K$ factors 
$\pi$.
% For the proof in characteristic different from two, we refer to \cite{Becher:qfconj}.

\begin{prop}\label{prop:char2extra}
Assume $\kar(F)=2$. Let  $(B,\tau)$ be a totally decomposable orthogonal $F$-algebra with involution,  $(Q,\gamma,h)$ be an  $F$-quaternion algebra with quadratic pair and  $(A,\s,f)= (B,\tau)\otimes (Q,\gamma,h)$. Then for any   field extension  $K/F$ such that $A_K$ is split, there exists a $1$-fold Pfister form  $\pi$  over $K$ such that $(A,\s,f)_K\simeq \Ad((\mathfrak{Pf}(B,\tau))_K\otimes \pi)$.
\end{prop}
\begin{proof}
Let $\varphi=\mathfrak{Pf}(B,\tau)$ and $\rho$ a quadratic form over $K$ with $(A,\s,f)_K\simeq \Ad(\rho)$.  Note that $\dim(\rho)=2\dim(\varphi)$.

Assume first that $(B,\tau)_K$ is metabolic. Then by  \cite[(A.5)]{Tignol:galcohomgps} we have that $(A,\s,f)_K$ is hyperbolic and by \cref{prop:hypiffquad} that $\rho$ is hyperbolic.
We may thus take $\pi$ to be the hyperbolic $2$-dimensional quadratic form. 

Assume now that $(B,\tau)_K$ is not metabolic. Then the bilinear Pfister form $\varphi_K$ is anisotropic by  \cref{thm:pfisterinvar}.
We consider its function field $L=K(\varphi_K)$. 
Since $\varphi_L$ is metabolic, it follows by \cref{thm:pfisterinvar} that $(B,\tau)_L$ is metabolic.
Hence, $(A,\s,f)_L$ is hyperbolic by \cite[(A.5)]{Tignol:galcohomgps} and therefore $\rho_L$ is hyperbolic by \cref{prop:hypiffquad}.
By \cref{prop:funcfieldbf}, there exists a non-trivial nonsingular  quadratic form $\pi'$ over $K$ such that $\rho_\an\simeq \varphi_K\otimes\pi'$. 
We have $$\dim(\varphi)\cdot \dim(\pi')=\dim(\rho_\an)\leqslant \dim(\rho)= 2\dim(\varphi)\,.$$
As $\kar(F)=2$,  by \cite[(7.32)]{Elman:2008} $\dim(\pi')$ is even.
It follows that  $\dim(\pi')=2$ and $\rho_\an\simeq \rho$.
In particular, $\pi'$ is similar to a $1$-fold Pfister form $\pi$ and $\rho\simeq\rho_{\an}\simeq \varphi_K\otimes\pi'$. Hence $\rho$ is similar to $\varphi_K\otimes\pi$.
 \end{proof}

\begin{cor} \label{cor:main}
Assume that $\kar(F)=2$. Let $n\in\mathbb{N}$, let $(A,\s,f)$ be a totally decomposable $F$-algebra with quadratic pair with $\deg(A)=2^{n+1}$ and let $K/F$ be a  field extension such that  $A_K$ is split. 
Then  there exists an 
$n$-fold bilinear Pfister form $\varphi$ over $K$ and an
$(n+1)$-fold Pfister form $\rho$ over $K$ such that
$\varphi_K$ factors $\rho$ and
 $(A,\s,f)_K\simeq\Ad(\rho)$.
\end{cor}
\begin{proof} For $n=0$, this is trivial. Otherwise, 
by \cref{totdecompqp}, there exists a totally decomposable $F$-algebra with orthogonal involution $(B,\tau)$ and an $F$-quaternion algebra with quadratic pair $(Q,\gamma,h)$ such that $(A,\s,f)\simeq (B,\tau)\otimes (Q,\gamma,h)$. The result then follows from \cref{prop:char2extra} with $\varphi=\mathfrak{Pf}(B,\tau)$.
\end{proof}
%
%
%
%
%\begin{remark}
%The hypothesis  that the $F$-algebra with quadratic pair $(A,\s,f)$ in (\cref{prop:char2extra}) is isomorphic to the tensor product of a totally decomposable $F$-algebra with orthogonal involution and an $F$-quaternion algebra with quadratic pair is not a restriction on $(A,\s,f)$ by  (\cref{totdecompqp}).
%\end{remark}

\begin{thm}\label{thm:pfisterfactquad}
Let $\rho$ be a nonsingular quadratic form over $F$ with $\dim(\rho)\geq 2$.
Then $\Ad(\rho)$ is totally decomposable if and only if $\rho$ is similar to a Pfister form.
%A split $F$--algebra with quadratic pair is totally decomposable if and only if it is isomorphic to $\Ad(\rho)$ for a Pfister form $\rho$ over $F$.
%Let $(Q_{0},\sigma_{0},f)$ be an $F$--quaternion algebra with quadratic pair.
%Let $n\in\mathbb{N}$ and let $(Q_1,\sigma_1),\ldots, (Q_n,\sigma_n)$ be $F$-quaternion algebras with involution, all assumed of orthogonal type in the case where $\kar(F)\neq 2$.
%  If $Q_0\otimes\ldots \otimes Q_n$  is  split, then $(\bigotimes_{i=1}^n (Q_i,\sigma_i))\otimes (Q_{0},\sigma_{0},f)\simeq \Ad(\rho)$ for a Pfister form $\rho$ over $F$.
\end{thm}
\begin{proof}
If $\rho$ is a Pfister form over $F$, then we can write $\rho\simeq\varphi\otimes\pi$ for a bilinear Pfister form $\varphi$ and a $1$-fold quadratic Pfister form $\pi$ over $F$. 
Then by \cref{prop:tensor} we have that  $\Ad(\rho)\simeq \Ad(\varphi)\otimes \Ad(\pi)$. The $F$-algebra with involution $\Ad(\varphi)$ is totally decomposable by \cref{lemma:tensorbf}. As $\Ad(\pi)$ is an $F$-quaternion algebra with quadratic pair, it follows that $\Ad(\rho)$ is totally decomposable.
Assume conversely that $\Ad(\rho)$ is totally decomposable.   If $\kar(F)\neq2$, then any quadratic pair is equivalent to an orthogonal involution and  thus the result corresponds to \cite[Thm.~1]{Becher:qfconj}. If $\kar(F)=2$, then 
 $\rho$ is similar to a Pfister form by \cref{cor:main}.
\end{proof}

Let $(A,\s,f)$ be a totally decomposable $F$-algebra with quadratic pair with $\mathrm{deg}(A)=2^m$ and let $K/F$ be a field extension such that $A_K$ is split.
 Let $\pi$ be the   $m$-fold Pfister form  over $K$  such that $(A,\s,f)_K\simeq\Ad(\pi)$.
In general, it is not possible to find
  an $(m-1)$-fold quadratic Pfister form $\rho$ over $F$ such that $\rho_K$ factors $\pi$.
  This is illustrated by  the following example, which  is a variation of an example in \cite[(3.9)]{parimala:pfisterinv}.
In particular, this example shows that 
\cref{cor:main} cannot be extended to cover fields of characteristic different from $2$, where quadratic and bilinear Pfister forms are equivalent. %

\begin{ex}\label{ex:countex} Let $n\in \mathbb{N}$ with $n\geqslant 4$. 
By  \cite[(38.4)]{Elman:2008} and its proof, there exists a field $F$ such that all
 $3$-fold quadratic Pfister forms over $F$ are hyperbolic and there exist
$F$-quaternion algebras $Q_1,\ldots, Q_n$ such that $A=Q_1\otimes_F\cdots\otimes_F Q_n$ is a division  $F$-algebra. In particular, we have  $\deg(A)=\mathrm{ind}(A)=2^n$. For $i= 1,\ldots, n$,  let $\gamma_i$ be  the canonical involution on $Q_i$ if $\kar(F)=2$ and an orthogonal involution on $Q_i$ if $\kar(F)\neq2$. 
We obtain a totally decomposable $F$-algebra with quadratic pair 
$$ (A,\s,f ) = (Q_1,\gamma_1)\otimes\cdots \otimes (Q_{n-1},\gamma_{n-1})\boxtimes (Q_n,\gamma_n) \,. $$

By \cite[(3.3) and \S2.4]{karpenko:gensplit}
there exists a field extension $K/F$ such that $A_K$ is split and  $(A,\s,f)_K\simeq \Ad(\rho)$ for some quadratic form $\rho$ over $K$ such that  $\mathrm{ind}(A)$ divides $i_W(\rho)$. As for any such $\rho$ we have that  $\dim(\rho) = 2^n=\mathrm{ind}(A)$ and $i_W(\rho)\leqslant \frac{1}{2}\dim(\rho)$, it follows that $i_W(\rho)=0$, that is, $\rho$ is anisotropic. In particular, $\rho$ is not factored by $\pi_K$ for any $(n-1)$-fold Pfister form $\pi$ over $F$, as all such $\pi$ are hyperbolic.
\end{ex}
%%
%%%
%\begin{remark}
%Let $F$ be  a  field of characteristic $2$ with the same properties as the field in  \cref{ex:countex}. Note that such fields cannot be used to construct   counter-examples  to \cref{prop:char2extra} in a similar way to \cref{ex:countex}, as 
%while all $3$-fold quadratic Pfister forms over $F$ are hyperbolic,   it does not follow that  all $3$-fold bilinear Pfister forms over $F$ are metabolic. Indeed, one can easily adapt the construction  of $F$ from \cite[(38.4)]{Elman:2008}  using the methods of \cite[(5.3)]{mammone:indexred}  in such a way that we get  $[F:F^2]=\infty$.   It then follows from \cite[(8.5)]{HoffmannLaghribi:qfpfisterneigbourc2} that there are anisotropic $m$-fold bilinear Pfister forms over $F$ for every $m\in\mathbb{N}$. 
% \end{remark}

\small{}

\end{document}